\documentclass[reqno]{amsart}
\usepackage{microtype}

\usepackage[shortlabels]{enumitem}
\setlist[enumerate]{label={(\arabic*)}}

\usepackage{amssymb}
\usepackage[dvipsnames,table]{xcolor}
\usepackage
[colorlinks=true,linkcolor=Maroon,citecolor=OliveGreen,backref]
{hyperref}
\usepackage[abbrev,shortalphabetic]{amsrefs}  
\usepackage{cleveref}

\usepackage{stmaryrd} 
\usepackage{tikz}
\usetikzlibrary{3d, cd}

\usepackage[norefs, nocites]{refcheck}

\makeatletter
\newcommand{\refcheckize}[1]{%
  \expandafter\let\csname @@\string#1\endcsname#1%
  \expandafter\DeclareRobustCommand\csname relax\string#1\endcsname[1]{%
    \csname @@\string#1\endcsname{##1}\@for\@temp:=##1\do{\wrtusdrf{\@temp}\wrtusdrf{{\@temp}}}}%
  \expandafter\let\expandafter#1\csname relax\string#1\endcsname
}
\makeatother
\refcheckize{\cref}
\refcheckize{\Cref}

\numberwithin{equation}{section}

\newtheorem{theorem}{Theorem}[section]
\newtheorem*{theorem*}{Theorem}
\newtheorem{lemma}[theorem]{Lemma}
\newtheorem{proposition}[theorem]{Proposition}
\newtheorem{corollary}[theorem]{Corollary}

\theoremstyle{definition}
\newtheorem{remark}[theorem]{Remark}
\newtheorem{definition}[theorem]{Definition}

\newtheorem{conjecture}[theorem]{Conjecture} \Crefname{conjecture}{Conjecture}{Conjectures}

\newtheoremstyle{named}{}{}{\itshape}{}{\bfseries}{.}{.5em}{#3}
\theoremstyle{named}

\newcommand{\eps}{\varepsilon}
\renewcommand{\phi}{\varphi}
\renewcommand{\bar}{\overline}

\renewcommand{\subset}{\subseteq}

\renewcommand\tbinom[2]{{\textstyle \binom{#1}{#2}}}

\newcommand\opr[1]{\operatorname{#1}}

\def\Z{\mathbf{Z}}

\def\Aut{\opr{Aut}}
\def\Iso{\opr{Iso}}

\def\nsgp{\trianglelefteq}
\def\Sym{\opr{Sym}}
\def\Alt{\opr{Alt}}

\newcommand\br[1]{{\left(#1\right)}}
\newcommand\floor[1]{\left\lfloor{#1}\right\rfloor}

\newcommand\pfrac[2]{\br{\frac{#1}{#2}}}

\renewcommand{\{}{\left\lbrace}
\renewcommand{\}}{\right\rbrace}

\def\X{\mathfrak{X}}

\def\S{\mathfrak{S}}
\def\J{\mathcal{J}}  
\def\H{\mathcal{H}}  
\newcommand\Ham[2]{\H(#2, #1)}
\def\Cam{\mathcal{C}}
\def\T{\mathcal{T}}  

\def\rank{\opr{rank}}

\usepackage{todonotes}

\begin{document}
    
    \title{Hamming sandwiches}
    
    \author{Sean Eberhard}
    \address{Sean Eberhard, Centre for Mathematical Sciences, Wilberforce Road, Cambridge CB3~0WB, UK}
    \email{eberhard@maths.cam.ac.uk}
    
    \thanks{SE has received funding from the European Research Council (ERC) under the European Union’s Horizon 2020 research and innovation programme (grant agreement No. 803711).
    }
    
    \begin{abstract}
        We describe primitive association schemes $\X$ of degree $n$ such that $\Aut(\X)$ is imprimitive and $|\Aut(\X)| \geq \exp(n^{1/8})$,
        contradicting a conjecture of Babai.
        This and other examples we give are the first known examples of nonschurian primitive coherent configurations (PCC) with more than a quasipolynomial number of automorphisms.
        
        Our constructions are ``Hamming sandwiches'', association schemes sandwiched between the $d$th tensor power of the trivial scheme and the $d$-dimensional Hamming scheme.
        We study Hamming sandwiches in general, and exhaustively for $d \leq 8$.
        
        We revise Babai's conjecture by suggesting that any PCC with more than a quasipolynomial number of automorphisms must be an association scheme sandwiched between a tensor power of a Johnson scheme and the corresponding full Cameron scheme.
        If true, it follows that any nonschurian PCC has at most $\exp O(n^{1/8} \log n)$ automorphisms.
    \end{abstract}
    
    \maketitle
    
    \section{Introduction}
    
    Coherent configurations were introduced as a means to study permutation groups $G \leq \Sym(V)$ through combinatorics.
    The definition abstracts the combinatorial properties of the set of \emph{orbital relations},
    the set of orbits of $G$ on $V^2$.
    What makes the theory notable is that, even though it is a purely combinatorial theory including for example strongly regular graphs, much of the theory of groups including character theory can be recovered.
    For this reason the theory of coherent configurations is sometimes described as ``group theory without groups''.
    
    The theory has a complicated history, with origins in multiple fields.
    Coherent configurations were reinvented several times, sometimes in a slightly restricted form, and always with a fresh name.
    Schur~\cite{schur} called them ``Stammringe'' (root rings),
    Wielandt~\cite{wielandt}*{Chapter~IV} called them ``$S$-rings'',
    Bose and Shimamoto~\cite{bose--shimamoto} called them ``association schemes'',
    Weisfeiler and Leman~\cite{WL} called them ``cellular algebras'',
    and finally Higman~\cite{higman} coined the now dominant term ``coherent configuration''.
    Naturally, terminology has been inconsistent, so it is important to clearly specify our language, which mostly follows Higman.
    
    \subsection{Definitions}
    
    Our definitions are mostly consistent with the recent notes of Chen and Ponomarenko~\cite{CP}.
    Other references include Bannai--Ito~\cite{bannai-ito} (see also the updated version \cite{BannaiBannaiItoTanaka+2021}) and Bailey~\cite{bailey}.
    A \emph{configuration} $\X$ on a finite set $V$ is a partition of $V^2$ into nonempty relations $R_1, \dots, R_r$ such that
    \begin{enumerate}
        \item some subset of $\X$ partitions the diagonal $\{(v, v) : v \in V\}$,
        \item if $R \in \X$ then $R^* = \{(x, y) : (y, x) \in R\} \in \X$.
    \end{enumerate}
    A configuration is  \emph{coherent} if
    \begin{enumerate}[resume]
        \item there are constants $p^i_{jk}$ (the \emph{intersection numbers} or \emph{structure constants}) for all $i, j, k$ such that
    \[
        \forall (u, v) \in R_i : ~
        \# \{ w \in V : (u, w) \in R_j , (w, v) \in R_k \} = p^i_{jk}.
    \]
    \end{enumerate}
    The \emph{degree} of $\X$ is the number of vertices $|V|$;
    the \emph{rank} of $\X$ is the number of relations $r$.

    A configuration is \emph{homogeneous} if, more strongly than (1), $\X$ contains the identity relation $\{(v, v) : v \in V\}$.
    A configuration $\X$ is \emph{symmetric} if $R^* = R$ for every $R \in \X$.
    A homogeneous symmetric coherent configuration is called an \emph{association scheme}.
    
    Even if $\X$ is homogeneous, it may be possible to break it into smaller pieces that can in principle be studied separately.
    Each relation $R \in \X$ apart from the identity relation defines a directed graph, called a \emph{constituent graph} or \emph{color graph} of $\X$.
    A homogeneous coherent configuration $\X$ is \emph{primitive} if each color graph is connected (there is no distinction between weak and strong connectedness, since coherence guarantees that each color graph is regular).
    
    There are two common notions of isomorphism for coherent configurations, depending on whether the relations of $\X$ are considered labeled.
    In this paper we usually consider $\X$ an \emph{unordered} partition $\{R_1, \dots, R_r\}$ of $V^2$, and thus we call two configurations $\X$ and $\X'$ \emph{isomorphic}, written $\X \cong \X'$, if there is a bijection $f : V \to V'$ mapping each $R \in \X$ to some $R' \in \X'$ (i.e., mapping $\X$ as a partition of $V^2$ to $\X'$ as a partition of $(V')^2$).
    Sometimes the term \emph{weakly isomorphic} is used instead, with the same meaning, to distinguish from another notion discussed in a moment.
    The set of weak isomorphisms $\X \to \X'$ is denoted $\Iso_w(\X, \X')$.
    The \emph{weak automorphism group} $\Aut_w(\X) = \Iso_w(\X, \X)$ thus consists of the permutations of $V$ which permute the relations of $\X$.

    If instead the labels of $R_1, \dots, R_r$ are considered part of the defining data of $\X$ then it is reasonable to say $\X$ and $\X'$ are isomorphic only if there is a bijection $f : V \to V'$ mapping $R_i$ to $R'_i$ for each $i$; this is the notion of \emph{strong isomorphism}.
    In general this notion makes sense only if $\X$ is equipped with extra data, namely a labeling of its relations, but we can always make sense of the \emph{strong automorphism group} $\Aut_s(\X) \nsgp \Aut_w(\X)$, which consists of permutations of $V$ preserving each relation individually.
    The group $\Aut_s(\X)$ is usually abbreviated to simply $\Aut(\X)$;
    its elements are usually just called \emph{automorphisms}.
    
    Whenever $\pi$ and $\pi'$ are partitions of the same set we write $\pi \leq \pi'$ if every cell of $\pi$ is contained in a cell of $\pi'$.
    In particular this applies to configurations $\X$ and $\X'$ on the same vertex set.
    If $\X \leq \X'$ then $\X'$ is a \emph{fusion} of $\X$,
    and $\X$ is a \emph{fission} of $\X'$.
    (This is opposite to the convention of \cite{CP}, but agrees with \cite{bailey}.)
    
    The prototypical coherent configuration is the orbital configuration of a permutation group. Suppose $G \leq \Sym(V)$ is a permutation group, and let $\X(G)$ be the partition of $V^2$ into $G$-orbits. Then $\X(G)$ is a coherent configuration. We call $\X(G)$ the \emph{orbital configuration} of $G$, and any configuration of the form $\X(G)$ is called \emph{schurian}
    (after Schur, who guessed that all coherent configurations were orbital configurations).
    Note that $G$ is transitive if and only if $\X(G)$ is homogeneous,
    and $G$ is primitive if and only if $\X(G)$ is primitive.
    Note also that $\X$ is schurian if and only if $\X = \X(\Aut(\X))$.
    
    Prominent among the schurian configurations are the \emph{discrete configuration} $\X(1)$, in which every relation is a singleton, and the \emph{trivial scheme} $\T_n = \X(\Sym(n))$, which is the unique configuration of degree $n$ and rank 2 (for $n > 1$).
    Other basic examples include the \emph{Hamming scheme}\footnote{%
    The author would rather swap the parameter order in $\Ham md$ to emphasize the comparison with $\J(m, d)$,
    but the notation is entrenched in the literature.}
     $\Ham{m}{d}$, whose vertex set is $[m]^d$, where $[m] = \{1, \dots, m\}$, and whose relations are defined by Hamming distance -- the orbital scheme of $\Sym(m) \wr \Sym(d) \leq \Sym([m]^d)$ --
    and the \emph{Johnson scheme} $\J(m, k)$, whose vertices are $k$-subsets of $[m]$ and whose relations are defined by $| u \setminus v | = i$ for $i \in \{0, \dots, k\}$ -- the orbital scheme of $\Sym(m)^{(k)} \leq \Sym(\binom{[m]}{k})$ (throughout the paper, we write $\Sym(\Omega)^{(k)}$ for the image of the natural map $\Sym(\Omega) \to \Sym(\binom{\Omega}{k})$).
    Note that $\J(m, k) \cong \J(m, m-k)$, so we may always assume $m \geq 2k$.
    
    If $\X$ and $\X'$ are coherent configurations on $V$ and $V'$ respectively, their \emph{tensor product} $\X \otimes \X'$ is the coherent configuration on $V \times V'$ whose relations are
    \[
        R \otimes R' = \{((u, u'), (v, v')) : (u, v) \in R, (u', v') \in R'\} \qquad (R \in \X, R' \in \X').
    \]
    If $\X$ is a coherent configuration and $G \leq \Sym(d)$, then $G$ acts on the tensor power $\X^d$ by weak automorphisms;
    the \emph{exponentiation} $\X \uparrow G$ is the coherent configuration obtained from $\X^d$ by fusing $G$-orbits (see \cite{CP}*{Section~3.4.2} or \cite{godsil}).
    For example, $\Ham{m}{d} \cong \T_m \uparrow \Sym(d)$.
    
    Association schemes of the form $\J(m, k) \uparrow G$ for some group $G \leq \Sym(d)$ are called \emph{Cameron schemes}.
    We refer to
    \[
        \Cam(m, k, d) = \J(m, k) \uparrow \Sym(d)
    \]
    as the \emph{full} Cameron scheme.
    Its vertices are $d$-tuples of $k$-subsets of $[m]$, and the cell containing the pair $(u, v) \in V^2$ is determined by $|u_1 \setminus v_1|, \dots, |u_d \setminus v_d|$, ignoring order.
    
    \subsection{The theorems of Babai and Cameron}
    
    If the purpose of the theory of primitive coherent configurations (PCCs) is to study permutation groups through combinatorics,
    then one of the landmark results of the field must be Babai's 1981 paper \cite{babai-uniprimitive} in which the following theorem was proved.
    
    \begin{theorem}[Babai]
        If $\X$ is a nontrivial PCC of degree $n$ then
        \[
            |\Aut(\X)| \leq \exp(4 n^{1/2} (\log n)^2).
        \]
    \end{theorem}
    
    In particular, if $G$ is a primitive permutation group of degree $n$, and not 2-transitive, then $\X(G)$ is a nontrivial PCC and $G \leq \Aut(\X(G))$, so
    \begin{equation}
        \label{eq:uniprimitive-bound}
        |G| \leq \exp(4 n^{1/2} (\log n)^2).
    \end{equation}
    The examples $\Sym(m) \wr C_2 \leq \Sym(m^2)$ and $\Sym(m)^{(2)} \leq \Sym(\binom{m}{2})$ show that \eqref{eq:uniprimitive-bound} cannot be improved by more than a log factor in the exponent.
    Moreover, it was proved by Babai and Pyber~\cites{babai-2-trans,pyber-2-trans} using related ideas that if $G$ is 2-transitive, and neither $\Sym(n)$ itself nor $\Alt(n)$, then
    \[
        |G| \leq \exp O((\log n)^3).
    \]
    Together these results solve one of the central problems of 19th century group theory
    -- bounding the order of a  primitive permutation group --
    by purely combinatorial methods.
    
    On the other hand, the classification of finite simple groups (CFSG) combined with the O'Nan--Scott theorem gives precise information, as in the following theorem of Cameron~\cite{cameron} (see also Liebeck~\cite{liebeck}, Liebeck--Saxl~\cite{liebeck--saxl}, and Mar\'oti~\cite{maroti}).
    
    \pagebreak[2]
    \begin{theorem}[Cameron, CFSG]
        \label{thm:cameron}
        If $G \leq \Sym(n)$ is primitive then either
        \begin{enumerate}
            \item there are integers $m, k, d$ such that $n = \binom{m}{k}^d$ and (up to conjugation)
            \[
                (\Alt(m)^{(k)})^d \leq G \leq \Sym(m)^{(k)} \wr \Sym(d),
            \]
            where the wreath product has the product action on $d$-tuples of $k$-sets
            or
            \item $|G| \leq \exp O( (\log n)^2)$.
        \end{enumerate}
    \end{theorem}
    
    The groups $G$ as in \Cref{thm:cameron}(1) are called \emph{Cameron groups}. The corresponding orbital schemes $\X(G)$ are precisely the Camerons schemes as defined in the previous section.
    
    While \Cref{thm:cameron} is a satisfying description of all the largest primitive permutation groups,
    it would be preferable to have a proof which is intelligible to mortals.
    The theory of coherent configurations, as a combinatorial generalization of group theory,
    provides an opportunity to make this desire precise.
    In this spirit the following conjecture has been advanced by Babai~\cite{babai-ICM2018}*{Conjecture~12.1} as part of the ``symmetry vs regularity'' project.
    
    \begin{conjecture}[Babai]
        \label{conj:babai}
            Let $\X$ be a PCC of degree $n$. Then either
            \begin{enumerate}
                \item $\X$ is a Cameron scheme or
                \item $|\Aut(\X)| \leq \exp O((\log n)^{O(1)})$.
            \end{enumerate}
    \end{conjecture}
    
    Essentially this conjecture has been repeated many times, sometimes faithfully~\cites{kivva-arxiv,sun-wilmes,wilmes-thesis}
    and sometimes in a weaker form either only predicting the primitivity of $\Aut(\X)$~\cites{babai-srg-1,babai-srg-2,sun-thesis,sun-wilmes-abs}
    or treating the minimal degree or thickness of $\Aut(\X)$ as more basic than its order~\cites{kivva1,kivva-arxiv}.
    Here the \emph{minimal degree} $\mu(G)$ of a permutation group $G$ is the minimal number of points moved by some nontrivial element, while the \emph{thickness} $\theta(G)$ is the largest $t$ such that $G$ has a section isomorphic to $\Alt(t)$.
        
    \begin{conjecture}[Babai]
        \label{conj:babai2}
        Let $\X$ be a PCC of degree $n$. Then either
        \begin{enumerate}
            \item $\Aut(\X)$ is primitive or
            \item $|\Aut(\X)| \leq \exp O((\log n)^{O(1)})$.
        \end{enumerate}
    \end{conjecture}
    
    \begin{conjecture}[Babai]
        \label{conj:babai3}
        Let $\X$ be a PCC of degree $n$. Then either $\X$ is a Cameron scheme or
        \begin{enumerate}
            \item $\mu(\Aut(\X)) \ge cn$ for some constant $c > 0$ and
            \item $\theta(\Aut(\X)) \le (\log n)^{O(1)}$.
        \end{enumerate}
    \end{conjecture}
    
    \Cref{conj:babai} is true for $|\Aut(\X)| > \exp(n^{1/3} (\log n)^C)$ by a result of Sun and Wilmes~\cite{sun-wilmes} (in this case, the only possibilities for $\X$ are the strongly regular graphs $\Ham{m}{2}$ and $\J(m, 2)$).
    \Cref{conj:babai3} is known in the bounded-rank distance-regular case by a result of Kivva~\cite{kivva1}, and in the case of PCCs of rank at most $4$ by \cites{babai-srg-1, kivva-arxiv}.
    
    \subsection{Sandwiches}
    
    The first purpose of this paper is to refute \Cref{conj:babai} as stated.
    However, the spirit of the conjecture is still reasonable,
    so the second purpose is to suggest a revision
    with similar consequences,
    and to develop some of the theory of the new larger class of exceptions.
    
    \begin{theorem}
        \label{thm:main}
        For every $m \geq 3$ there is a primitive association scheme $\X$ of degree $n = m^8$ and rank $28$ such that
        $\Aut(\X)$ is imprimitive and $|\Aut(\X)| = 8m!^8 \geq \exp(n^{1/8})$.
        Also $\mu(\Aut(\X)) = 2m^7 = 2n^{7/8}$ and $\theta(\Aut(\X)) = m = n^{1/8}$.
    \end{theorem}
    
    Since $\Aut(\X)$ is imprimitive, $\X$ is in particular nonschurian and hence not a Cameron scheme, so this example contradicts \Cref{conj:babai}.
    It also contradicts the weaker variants \Cref{conj:babai2} and \Cref{conj:babai3}.
    (\Cref{thm:main} does not contradict the results quoted above on special cases of \Cref{conj:babai} and \Cref{conj:babai3}, since $|\Aut(\X)|$ is much smaller than $\exp(n^{1/3})$, $\X$ is not distance-regular, and $\rank(\X) = 28$.)
    
    However, $\X$ still closely resembles a Cameron scheme, apart from being nonschurian.
    In particular it is sandwiched between a tensor power of the trivial scheme and the Hamming scheme:
    \begin{equation}
        \label{eq:hamming-sandwich}
        \T_m^d \leq \X \leq \Ham{m}{d},
    \end{equation}
    where $d = 8$.
    We call any association scheme $\X$ obeying \eqref{eq:hamming-sandwich} for some $m, d \geq 1$ a \emph{Hamming sandwich}.
    The possibility that there could be such a sandwiched scheme which is \emph{nonschurian} has been overlooked.
    We will show that if $m$ is sufficiently large depending on $d$ then there is a bijective correspondence between Hamming sandwiches and certain partitions of $2^{[d]}$, which we call \emph{set association schemes}, a new class of combinatorial objects studied for the first time in this paper.
    The proof of \Cref{thm:main} is then completed by exhibiting a
    set association scheme of degree $d = 8$ with suitable properties.
    
    The structure of these examples suggests a revision of \Cref{conj:babai}.
    Let us call any association scheme $\X$ such that
    \[
        \J(m, k)^d \leq \X \leq \Cam(m, k, d)
    \]
    for some $m, k, d \geq 1$ a \emph{Cameron sandwich}.
    As for Hamming sandwiches, there is a bijective correspondence between Cameron sandwiches and partitions of $\{0, \dots, k\}^d$ obeying a certain coherence condition.
    All such schemes have large automorphism group, though they may not be schurian, which suggests the following revision of \Cref{conj:babai}, in which only ``Cameron scheme'' has been replaced with ``Cameron sandwich''.
    
    \pagebreak[2]
    \begin{conjecture}[sandwich conjecture]
        \label{conj:sand}
        Let $\X$ be a PCC of degree $n$. Then either
        \begin{enumerate}
            \item $\X$ is a Cameron sandwich or
            \item $|\Aut(\X)| \leq \exp O((\log n)^{O(1)})$.
        \end{enumerate}
    \end{conjecture}
    
    If true, there are notable consequences.
    For example, since $\Cam(m, k, d)$ has rank $\binom{k+d}{k} > 3$ unless $k = d = 1$ or $\{k, d\} = \{1, 2\}$, it follows immediately that any nontrivial strongly regular graph has quasipolynomially many automorphisms,
    excepting only the constituents of $\Ham{m}{2}$ and $\J(m, 2)$.
    Moreover this extends to distance-regular graphs of bounded diameter $d \geq 3$, with the exception of Hamming and Johnson graphs, by the result of Kivva~\cite{kivva1}*{Theorem~1.11}.
    
    Similarly, it follows from \Cref{conj:sand} and a result of Wilmes~\cite{wilmes-thesis}*{Theorem~8.2.1} that
    if $G \leq \Sym(n)$ then either $G$ is a Cameron group or $|G|$ is at most quasipolynomial.
    The result of Wilmes does not depend on the CFSG, so a CFSG-free proof of \Cref{conj:sand} would entail a CFSG-free proof of a slightly weaker (but still quasipolynomial) version of Cameron's theorem.
    
    In another direction, \Cref{conj:sand} has consequences for the graph isomorphism problem (GI).
    Famously, Babai recently gave a quasipolynomial-time GI algorithm~\cite{babai-GI}.
    The validity of \Cref{conj:sand} implies that a comparatively simple GI algorithm already has quasipolynomial running time: see \cite{babai-GI}*{Remark~6.1.3}.
    
    We give two more applications of \Cref{conj:sand}, which will be proved in the body of the paper.
    The first result asserts that the exponent $1/8$ in \Cref{thm:main} is maximal.
    The second concerns the rank.
    
    \begin{proposition}
        \label{prop:nonschurian1/8}
        Assume \Cref{conj:sand}.
        Let $\X$ be a nonschurian PCC.
        Then
        \[
            |\Aut(\X)| \leq \exp O(n^{1/8} \log n).
        \]
    \end{proposition}
    
    Among the counterexamples constructed in this paper,
    the minimal rank of a counterexample to \Cref{conj:babai2} is $28$ (as in \Cref{thm:main}),
    while the minimal rank of a counterexample to \Cref{conj:babai} is $24$ (see \Cref{table:nonschurian9}).
    It is still unknown whether there is a rank-$3$ counterexample to \Cref{conj:babai} or a rank-$5$ counterexample to \Cref{conj:babai3}, but if \Cref{conj:sand} is true then the rank of any nonschurian PCC with more than a quasipolynomial number of automorphisms must be at least $12$.
    Further work on set association schemes will likely increase this lower bound.
    
    \begin{proposition}
        \label{prop:ranks}
        Assume \Cref{conj:sand}.
        Let $\X$ be a nonschurian PCC of rank less than $12$.
        Then $|\Aut(\X)| \leq \exp O((\log n)^{O(1)}).$
    \end{proposition}
    
    \Cref{conj:sand} is true for $|\Aut(\X)| > \exp(n^{1/3} (\log n)^C)$ by \cite{sun-wilmes}.
    For smaller values of $|\Aut(\X)|$ it is open.
    In a follow-up paper, the author and Muzychuk~\cite{followup} will prove \Cref{conj:sand} within the class of fusions of tensor powers of Johnson schemes with bounded parameters.
    To be precise, let $\X$ be a PCC such that $\J(m, k)^d \leq \X$, where $m \geq m_0(k, d)$.
    Then either
    \begin{enumerate}
        \item $\J(m, k)^d \leq \X \leq \Cam(m, k, d)$, or
        \item $\T_{\binom{m}{k}^e}^{d/e} \leq \X \leq \Ham{\tbinom{m}{k}^e}{d/e}$ for some integer $e \mid d$.
    \end{enumerate}
    In other words, all ``open sandwiches'' are sandwiches.
    
    Finally, let us mention one more conjecture, which has been emphasized by Pyber (private communication).
    This conjecture is significantly weaker than \Cref{conj:sand}, but seems to represent a key challenge (compare with \Cref{conj:babai2}).
    
    \begin{conjecture}[transitivity conjecture]
        Let $\X$ be a PCC of degree $n$. Then either
        \begin{enumerate}
            \item $\Aut(\X)$ is transitive or
            \item $|\Aut(\X)| \le \exp O((\log n)^{O(1)})$.
        \end{enumerate}
    \end{conjecture}
    
    \subsection{Acknowledgments}
    
    It is pleasure to thank Laci Pyber, Misha Muzychuk, and Laci Babai for numerous invaluable discussions.

    \section{Set association schemes}
    
    Let $\Omega$ be a finite set and let $2^\Omega$ denote the power set of $\Omega$.
    The \emph{type} $\tau$ of a triple $(a, b, c) \in (2^\Omega)^3$ is its orbit $(a, b, c)^{\Sym(\Omega)}$ under the action of $\Sym(\Omega)$. Concretely, the type of a triple $(a, b, c)$ is determined by the vector of intersection sizes
    \[
        (|a|, |b|, |c|, |a \cap b|, |a \cap c|, |b \cap c|, |a \cap b \cap c|).
    \]
    We call $(a, b, c)$ a \emph{triangle} if $a \subset b \cup c$ and $b \subset a \cup c$ and $c \subset a \cup b$.
    The type of a triangle will be called a \emph{triangle type}.
    
    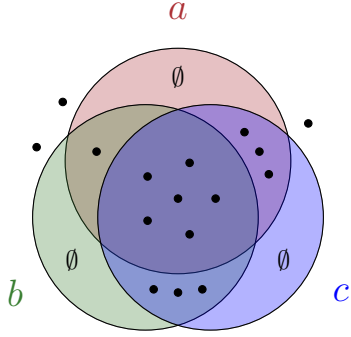
\begin{figure}[ht]
        \centering
        \begin{tikzpicture}[scale=0.5]
                \foreach \color/\rotation in {Maroon/90, OliveGreen/90+120, blue/90+240}{
                    \draw[fill=\color, opacity=0.3] (\rotation:1) circle (3) ;
                    \draw (\rotation:1) circle (3);
                }
                \node[color=Maroon] at (90:5) {\LARGE $a$};
                \node[color=OliveGreen] at (90+120:5) {\LARGE $b$};
                \node[color=blue] at (90+240:5) {\LARGE $c$};
                
                \foreach \rotation in {90, 90+120, 90+240}
                    \node at (\rotation:3.25) {$\emptyset$};
                
                \foreach \position in {(0:0), (0:1), (72:1), (144:1), (216:1), (288:1),
                                                    (30:2.5), (30+15:2.5), (30-15:2.5),
                                                    (-90:2.5), (-90+15:2.5), (-90-15:2.5),
                                                    (150:2.5),
                                                    (90-60:4), (90+70:4), (90+50:4)
                                                    }
                    \node[draw, circle, inner sep=1pt, fill] at \position {};
        \end{tikzpicture}
        \caption{A typical triangle $(a, b, c)$}
    \end{figure}
    
    Let $\S$ be a partition of $2^\Omega$ and let $\tau$ be a triple type.
    For $\alpha, \beta, \gamma \in \S$ and $a \in \alpha$ we denote
    \[
        p^a_{\beta\gamma;\tau} = \# \{(b, c) \in \beta \times \gamma : (a, b, c) \in \tau\}.
    \]
    We say $\S$ is \emph{coherent} with respect to the type $\tau$ if $p^a_{\beta\gamma;\tau} = p^{a'}_{\beta\gamma;\tau}$ for all $\alpha, \beta, \gamma \in \S$ and all $a, a' \in \alpha$;
    in this case we denote the common value of $p^a_{\beta\gamma;\tau}$ by $p^\alpha_{\beta\gamma;\tau}$.
    The numbers $p^\alpha_{\beta\gamma;\tau}$ are the \emph{structure constants} of $\S$.
    
    \begin{definition}
        \label{defn:SAS}
        A \emph{set association scheme} is a partition $\S$ of $2^\Omega$ such that
        \begin{enumerate}
            \item $a, a' \in \alpha \in \S \implies |a| = |a'|$,
            \item $\S$ is coherent with respect to all triangle types.
        \end{enumerate}
        We call $\S$ \emph{fully coherent} if $\S$ is coherent with respect to all types (triangle or not).
        The \emph{degree} of $\S$ is $|\Omega|$.
        The \emph{rank} of $\S$ is the number of cells, denoted $\rank(\S)$.
    \end{definition}
    
    Like coherent configurations, the prototypical set association scheme is the orbital scheme of a permutation group.
    Let $G \leq \Sym(\Omega)$ and let $\S(G) = 2^\Omega / G$ be the partition of $2^\Omega$ into $G$-orbits.
    Then $\S(G)$ is a fully coherent set association scheme.
    We call $\S(G)$ the \emph{orbital scheme} of $G$.
    A set association scheme is called \emph{schurian} if is the orbital scheme of some group, and \emph{nonschurian} otherwise.
    Prominent among the schurian schemes are the \emph{discrete set association scheme}
    \[
        \S(1) = \{\{a\} : a \subset \Omega\}
    \]
    and the \emph{trivial set association scheme}
    \[
        \S(\Sym(\Omega)) = \{\{a \subset \Omega : |a| = k\} : 0 \leq k \leq |\Omega|\}.
    \]
    
    For $a \in 2^\Omega$ we use $a^c$ to denote the complement $\Omega \setminus a$.
    For $\alpha \subset 2^\Omega$ we use $\alpha^c$ to denote the set of complements $\{a^c : a \in \alpha\}$ (not the complement $2^\Omega \setminus \alpha$).
    
    \begin{lemma}
        \label{lem:basic-props}
        Let $\S$ be a set association scheme. Then
        \begin{enumerate}
            \item $\alpha \in \S \implies\alpha^c \in \S$,
            \item for any $\alpha, \beta \in \S$, the containment graph $\{(a, b) \in \alpha \times \beta : a \subset b\}$ is biregular.
        \end{enumerate}
    \end{lemma}
    \begin{proof}
        Let $\omega = \{\Omega\}$. By condition (1) of \Cref{defn:SAS}, $\omega \in \S$.
        
        (1) Let $a \in \alpha \in \S$ and let $\beta$ be the part of $\S$ containing $a^c$.
        Let $\tau$ be the type of the triangle $(a, a^c, \Omega)$.
        For any $a' \in \alpha$,
        \[
            \#\{(b, c) \in \beta \times \omega : (a', b, c) \in \tau\} =
            \begin{cases}
                1 & \text{if}~ (a')^c \in \beta \\
                0 & \text{else}.
            \end{cases}.
        \]
        Thus by coherence $(a')^c \in \beta$ for every $a' \in \alpha$, so $\alpha^c \subset \beta$, and likewise $\beta^c \subset \alpha$, so $\beta = \alpha^c$.
        
        (2) Let $(a_0, b_0) \in \alpha \times \beta$ be such that $a_0 \subset b_0$ and let $\tau$ be the type of the triangle $(a_0, b_0, b_0)$. Then for any $a \in \alpha$,
        \[
            \# \{(b, b') \in \beta \times \beta : (a, b, b') \in \tau\} = \# \{ b \in \beta : a \subset b\},
        \]
        and by coherence this must be the same for every $a \in \alpha$.
        By a similar argument, or using (1), the quantity $\# \{a \in \alpha : a \subset b\}$ must be the same for every $b \in \beta$.
        Hence the containment graph is biregular.
    \end{proof}
    
    Now let $m \geq 1$ and let $V = [m]^\Omega$.
    If $\S$ is a partition of $2^\Omega$ we define $[m]^\S$ to be the symmetric configuration on $V$ with relations $R_\alpha$ ($\alpha \in \S$) defined by
    \[
        (u, v) \in R_\alpha \iff \{i : u_i \neq v_i\} \in \alpha.
    \]
    Note that if $\Omega = [d]$ and $\S = \S(1)$ then $[m]^\S = \T_m^d$, while if $\S = \S(\Sym(\Omega))$ then $[m]^\S = \Ham{m}{d}$.
    Now if $\S$ is an arbitrary set association scheme then $\S(1) \leq \S \leq \S(\Sym(\Omega))$, so
    \[
        \T_m^d \leq [m]^\S \leq \Ham{m}{d}.
    \]
    
    \begin{proposition}
        \label{lem:hamming-sandwiches}
        Let $m, d \geq 1$ and let $\Omega = [d]$.
        \begin{enumerate}
            \item If $\S$ is a set association scheme on $\Omega$ then $\X = [m]^\S$ is an association scheme
            such that $\T_m^d \leq \X \leq \Ham{m}{d}$.
            \item Conversely if $\T_m^d \leq \X \leq \Ham{m}{d}$ and $m \geq m_0(d)$ then $\X = [m]^\S$ for some set association scheme $\S$ on $\Omega$. One can take $m_0(d) = 3^d + 4$.
        \end{enumerate}
    \end{proposition}
    
    \begin{proof}
        (1)
        Let $\X = [m]^\S$.
        Then $\X$ is a configuration because $\{\emptyset\} \in \S$,
        and plainly $\X$ is homogeneous and symmetric.
        We must check coherence.
        Let $u, v \in V$.
        Let $\alpha$ be the cell of $\S$ containing the set $a = \{i : u_i \neq v_i\}$.
        If $b, c \subset \Omega$ then the the number of $w \in V$ such that $\{i : u_i \neq w_i\} = b$ and $\{i : w_i \neq v_i\} = c$ is zero unless $(a, b, c)$ is a triangle, and in this case it is
        \[
            (m-1)^{|b \cap c \setminus a|} (m-2)^{|a \cap b \cap c|}.
        \]
        Hence the number of $w \in V$ such that $\{i : u_i \neq w_i\} \in \beta$ and $\{i: w_i \neq v_i\} \in \gamma$ is
        \begin{equation}
            \label{eq:coherence-sum}
            p^a_{\beta\gamma}(m) = \sum_{\substack{b \in \beta, c \in \gamma \\ (a, b, c) ~ \text{triangle}}} (m-1)^{|b \cap c \setminus a|} (m-2)^{|a \cap b \cap c|}.
        \end{equation}
        Since the summand depends only on the type of $(a, b, c)$, we can write this as a sum over types:
        \[
            p^a_{\beta\gamma}(m) = \sum_{\tau~\text{triangle type}} p^a_{\beta\gamma;\tau} (m-1)^{\tau_{011} - \tau_{111}} (m-2)^{\tau_{111}}.
        \]
        Here $\tau_{011} = |b \cap c|$ and $\tau_{111} = |a \cap b \cap c|$ for any $(a, b, c) \in \tau$.
        By coherence of $\S$, this depends on $a$ only through $\alpha$, so this proves coherence of $\X$.
        
        (2)
        Conversely,  suppose $\X$ is an association scheme such that $\T_m^d \leq \X \leq \Ham md$.
        Since the relations of $\T_m^d$ are in bijection with $2^\Omega$
        and $\T_m^d$ refines $\X$,
        there is some partition $\S$ of $2^\Omega$ such that $\X = [m]^\S$.
        Since $\X \leq \Ham md$, part (1) of \Cref{defn:SAS} holds.
        Coherence of $\X$ demands that the sum \eqref{eq:coherence-sum} depends only on the cell $\alpha \in \S$ containing $a$. Therefore for any $a, a' \in \alpha$ we must have equality
        \[
            p^a_{\beta\gamma}(m) = p^{a'}_{\beta\gamma}(m).
        \]
        Now the type of a triangle $(a, b, c)$ is determined by $|a|$, $|b|$, $|c|$, and $|a \cap b \cap c|$, because
        $a \cup b = a \cup c = b \cup c = a \cup b \cup c$ and
        \[
            |a| + |b| + |c| = 2 |a \cup b \cup c| + |a \cap b \cap c|.
        \]
        Suppose $(|a|, |b|, |c|) = (i, j, k)$ and let $N_\ell^a$ be the number of $(b, c) \in \beta \times \gamma$ with $(a, b, c)$ a triangle and $|a \cap b \cap c| = \ell$.
        For each such triangle we have
        \[
            |b \cap c \setminus a| = |b \cap c| - \ell = (j + k - i - \ell)/2,
        \]
        so
        \[
            p^a_{\beta\gamma}(m) = \sum_{\ell \geq 0} N_\ell^a (m-1)^{(j + k - i - \ell)/2} (m-2)^\ell,
        \]
        and $N_\ell^a = 0$ if $j + k - i - \ell$ is odd or negative.
        If $\S$ is not coherent then there is some $N_\ell^a \neq N_\ell^{a'}$ for some $\alpha, \beta, \gamma \in \S$ and some $a, a' \in \alpha$.
        Let $L$ be the largest such index $\ell$.
        Then
        \[
            0 = p^a_{\beta\gamma}(m) - p^{a'}_{\beta\gamma}(m)
            = \sum_{\ell = 0}^L (N^a_\ell - N^{a'}_\ell) (m-1)^{(j+k-i-\ell)/2} (m-2)^\ell.
        \]
        Dividing by $(m-1)^{(j+k-i-L)/2} (m-2)^L$ and rearranging,
        \[
            N_L^{a'} - N_L^a = \sum_{\substack{\ell=0 \\ \ell - L ~ \text{even}}}^{L-1} (N_\ell^a - N_\ell^{a'}) (m-1)^{(L-\ell)/2} (m-2)^{\ell-L}.
        \]
        Now we claim $N_\ell^a \leq 3^d$ uniformly, so
        \[
            1 \leq |N_L^{a'} - N_L^a|
            \leq 3^d \sum_{t > 0} \pfrac{m-1}{(m-2)^2}^t
            = 3^d / ((m-2)^2 / (m-1) - 1).
        \]
        Rearranging gives $m < 3^d + 4$.
        
        To prove the bound $N_\ell^a \leq 3^d$, consider choosing $b, c \subset \Omega$ such that $(a, b, c)$ is a triangle.
        Each element of $a$ must be in either $a \cap b \setminus c$, $a \cap c \setminus b$, or $a \cap b \cap c$.
        Each element of $a^c$ must be in either $(a \cup b \cup c)^c$ or $b \cap c \setminus a$.
        This proves $N_\ell^a \leq 3^{|a|} 2^{|a^c|} \leq 3^d$.
    \end{proof}
    
    In the next lemma we relate properties of $[m]^\S$ to properties of $\S$.
    Let $\S$ be a set association scheme with vertex set $\Omega$.
    We say $\S$ is \emph{homogeneous} if $\{a \in 2^\Omega : |a| = 1\}$ is a cell of $\S$.
    Two set association schemes $\S$ and $\S'$ (with vertex classes $\Omega$ and $\Omega'$) are \emph{(weakly) isomorphic} if there is a bijection $f : \Omega \to \Omega'$ mapping the partition $\S$ to the partition $\S'$.
    The \emph{weak automorphism group} of $\S$ is $\Aut_w(\S) = \Iso_w(\S, \S)$;
    its normal subgroup consisting of bijections $f : \Omega \to \Omega$ preserving each cell of $\S$ individually is the \emph{(strong) automorphism group} $\Aut(\S)$.
    We have already said $\S$ is called \emph{schurian} if it is the orbital scheme of some group;
    this is equivalent to $\S = \S(\Aut(\S))$.
    
    \begin{lemma}
        \label{lem:prop-correspondence}
        Let $\S$ be a set association scheme with vertex set $\Omega = [d]$, where $d \geq 2$, and let $m \geq 2$.
        \begin{enumerate}
            \item $\rank([m]^\S) = \rank(\S)$.
            \item $[m]^\S$ is primitive $\iff$ $\S$ is homogeneous and $m \geq 3$.
            \item $\Aut([m]^\S) = \Sym(m) \wr \Aut(\S)$.
            \item $[m]^\S$ is schurian $\iff$ $\S$ is schurian.
            \item $\Aut([m]^\S)$ is primitive $\iff$ $\Aut(\S)$ is transitive and $m \geq 3$.
        \end{enumerate}
    \end{lemma}
    \begin{proof}
        (1) Immediate from the definition of $[m]^\S$.
        
        (2) Suppose $[m]^\S$ is primitive. Let $a \in \alpha \in \S$ where $a$ is a singleton.
        By part (1) of \Cref{defn:SAS}, every $a' \in \alpha$ is a singleton.
        Since the color graph of $[m]^\S$ defined by $\alpha$ must be connected, we must have $\alpha = \{a \in 2^\Omega : |a| = 1\}$.
        Hence $\S$ is homogeneous.
        Moreover $[m]^\S \leq \Ham{m}{d}$, so $\Ham{m}{d}$ is primitive, so $m \geq 3$ (the maximal-distance graph of $\Ham{2}{d}$ is disconnected).
        
        Conversely, suppose $\S$ is homogeneous and $m \geq 3$.
        Let $\alpha = \{a \in 2^\Omega : |a| = 1\}$ and let $\beta \in \S \setminus \{\{\emptyset\}\}$.
        Then there is some $a \in \alpha$ and $b \in \beta$ such that $a \subset b$.
        By \Cref{lem:basic-props}(2) it follows that $\bigcup \beta = \Omega$.
        Now it is easy to prove using this and $m \geq 3$ that the constituent graph of $[m]^\S$ defined by $\beta$ is connected.
        
        (3) Let $A = \Aut([m]^\S)$. Since $\T_m^d \leq [m]^\S \leq \Ham{m}{d}$,
        \[
            \Sym(m)^d = \Aut(\T_m^d) \leq A \leq \Aut(\Ham{m}{d}) = \Sym(m)^d \Sym(d)^\eps,
        \]
        where $\eps : \Sym(d) \to \Sym([m]^d)$ is the embedding of $\Sym(d)$ as the subgroup of $\Sym([m]^d)$ permuting the axes. Hence by the modular property of groups
        \[
            A = A \cap (\Sym(m)^d \Sym(d)^\eps) = \Sym(m)^d (A \cap \Sym(d)^\eps).
        \]
        Now $A \cap \Sym(d)^\eps$ is the group of permutations of the axes preserving the relations of $[m]^\S$, so $A \cap \Sym(d)^\eps = \Aut(\S)^\eps$.
        
        (4) If $G \leq \Sym(d)$ then the orbital configuration of $\Sym(m) \wr G$ is $[m]^{\S(G)}$. Hence by (3) the orbital configuration of $\Aut([m]^\S)$ is $[m]^{\S(\Aut(\S))}$, and this is equal to $[m]^\S$ if and only if $\S = \S(\Aut(\S))$.
        
        (5) As in (4), $\X(\Aut([m]^\S)) = [m]^{\S(\Aut(\S))}$. Now apply (2).
        (Alternatively, see~\cite{dixon--mortimer}*{Lemma~2.7A}.)
    \end{proof}
    
    Finally we give two ways of constructing new set association schemes from old ones.
    Let $\S$ and $\S'$ be set association schemes on $\Omega$ and $\Omega'$.
    Let $\Omega \sqcup \Omega'$ denote the disjoint union of $\Omega$ and $\Omega'$.
    For $\alpha \in \S$ and $\alpha' \in \S'$, define
    \[
        \alpha \oplus \alpha' = \{a \sqcup a' : a \in \alpha, a' \in \alpha'\}.
    \]
    The \emph{direct sum} of $\S$ and $\S'$ is the partition of $2^{\Omega \cup \Omega'}$ defined by
    \[
        \S \oplus \S' = \{\alpha \oplus \alpha' : \alpha \in \S, \alpha' \in \S'\}.
    \]
    Note that if $G$ and $H$ are permutation groups acting on $\Omega$ and $\Omega'$ respectively then $\S(G \times H) = \S(G) \oplus \S(H)$.
    
    \begin{lemma}
        Let $\S$ and $\S'$ be set association schemes.
        \begin{enumerate}
            \item $\S \oplus \S'$ is a set association scheme.
            \item $\rank(\S \oplus \S') = \rank(\S) \rank(\S')$.
            \item $\Aut(\S \oplus \S') = \Aut(\S) \times \Aut(\S')$.
            \item $\S \oplus \S'$ is schurian $\iff$ $\S$ is schurian and $\S'$ is schurian.
            \item For $m \geq 1$, $[m]^{\S \oplus \S'} \cong [m]^\S \otimes [m]^{\S'}$.
        \end{enumerate}
    \end{lemma}
    \begin{proof}
    (1) Clearly cells determine size. We must check coherence.
    Let $\tau$ be a triangle type, let $\alpha\oplus\alpha', \beta\oplus\beta', \gamma\oplus\gamma' \in \S\oplus\S'$,
    and let $a \sqcup a' \in \alpha \oplus \alpha'$.
    Then the number of $(b\sqcup b', c\sqcup c') \in (\beta \oplus \beta') \times (\gamma \oplus \gamma')$ such that $(a \sqcup a', b \sqcup b', c \sqcup c') \in \tau$
    is
    \[
        \sum_{\tau_1, \tau_2} p^\alpha_{\beta\gamma;\tau_1} p^{\alpha'}_{\beta'\gamma';\tau_2},
    \]
    where the sum goes over all pairs of triangle types $\tau_1$ on $\Omega$ and $\tau_2$ on $\Omega'$ such that if $(a, b, c) \in \tau_1$ and $(a', b', c') \in \tau_2$ then $(a \sqcup a', b \sqcup b', c \sqcup c') \in \tau$.
    Thus we have coherence.
    
    (2) Clear.
    
    (3) Let $\pi \in \Aut(\S \oplus \S')$. Since $\pi$ fixes the cell $\Omega \sqcup \emptyset$, $\pi$ respects the partition of $\Omega \sqcup \Omega'$ into $\Omega$ and $\Omega'$,
    and clearly the set of $\pi \in \Aut(\S \oplus \S')$ respecting this partition is $\Aut(\S) \times \Aut(\S')$.
    
    (4) If $\S = \S(G)$ and $\S' = \S(H)$ then $\S \oplus \S' = \S(G \times H)$. On the other hand if $\S \oplus \S'$ is schurian then, by (3),
    \[
        \S \oplus \S' = \S(\Aut(\S \oplus \S')) = \S(\Aut(\S) \times \Aut(\S')) = \S(\Aut(\S)) \oplus \S(\Aut(\S')),
    \]
    which implies $\S = \S(\Aut(\S))$ and $\S' = \S(\Aut(\S'))$.
    
    (5) A definition-chase.
    \end{proof}
    
    A nontrivial direct sum $\S \oplus \S'$ is never homogeneous, whereas we are principally interested in homogeneous set association schemes.
    Another construction preserves homogeneity.
    Let $\S^k$ be the direct sum of $k$ copies of $\S$.
    If $G \leq \Sym(k)$ then $G$ permutes the cells of $\S^k$.
    Fusing the orbits of $G$ produces a new set association scheme which we call the \emph{wreath product} $\S \wr G$.
    The term is reasonable since if $H \leq \Sym(\Omega)$ then $\S(H) \wr G = \S(H \wr G)$, where $H \wr G \leq \Sym(\Omega \times [k])$ has the imprimitive action.
    
    \begin{lemma}
    \label{lem:wreath-product}
    Let $\S$ be a set association scheme and $G \leq \Sym(k)$, $k \geq 1$.
        \begin{enumerate}
            \item $\S \wr G$ is a set association scheme.
            \item $\rank(\S \wr G) = |[\rank(\S)]^k / G|$.
            \item $\S \wr G$ is homogeneous $\iff$ $\S$ is homogeneous and $G$ is transitive.
            \item $\Aut(\S \wr G) = \Aut(\S) \wr \bar G$, where $G \leq \bar G \leq \Sym(k)$.
            \item $\S \wr G$ is schurian $\iff$ $\S$ is schurian.
            \item For $m \geq 1$, $[m]^{\S \wr G} \cong [m]^\S \uparrow G$.
        \end{enumerate}
    \end{lemma}
    \begin{proof}
        (1) Routine.
        
        (2) Let $\S = \{\alpha_1, \dots, \alpha_r\}$, where $r = \rank(\S)$. Then the cells of $\S^k$ are
        ${\alpha_{i_1} \sqcup \cdots \sqcup \alpha_{i_k}}$ for $i_1, \dots, i_k \in [r]$.
        The $G$-orbit of ${\alpha_{i_1} \sqcup \cdots \sqcup \alpha_{i_k}}$ is the union of all ${\alpha_{j_1} \sqcup \cdots \sqcup \alpha_{j_k}}$ for $(j_1, \dots, j_k) \in (i_1, \dots, i_k)^G$.
        Hence $\S \wr G$ has $|[r]^k / G|$ cells.
        
        (3) Clear.
        
        (4) Let $A = \Aut(\S \wr G)$, and first suppose $G = \Sym(k)$.
        Write the vertex set of $\S^k$ as $\Omega_1 \cup \cdots \cup \Omega_k$,
        where $\Omega_1, \dots, \Omega_k$ are the disjoint copies of $\Omega$.
        Then $\{\Omega_1\}, \dots, \{\Omega_k\}$ are cells of $\S^k$
        and $\pi = \{\Omega_1, \dots, \Omega_k\}$ is a cell of $\S \wr \Sym(k)$,
        so $A$ preserves the partition $\pi$.
        Let $B \leq A$ be the subgroup of $A$ preserving each $\Omega_i$ individually.
        Then clearly $A = B \Sym(k)$, and $B = \Aut(\S^k) = \Aut(\S)^k$, since $B$ must preserve each cell of $\S \wr \Sym(k)$ of the form $\alpha \oplus \cdots \oplus \alpha$ for $\alpha \in \S$.
        Now for $G$ arbitrary we have
        \[
            B = \Aut(\S^k) \leq A \leq \Aut(\S \wr \Sym(k)) = B \Sym(k),
        \]
        so by the modular property of groups we have
        \[
            A = B (A \cap \Sym(k)) = \Aut(\S) \wr (A \cap \Sym(k)).
        \]
        The intersection $\bar G = A \cap \Sym(k)$ is the subgroup of $\Sym(k)$ preserving each orbit of $G$ in $[\rank(\S)]^k$.
        
        (5) If $\S = \S(H)$ then $\S \wr G = \S(H \wr G)$. On the other hand if $\S \wr G$ is schurian, then, by (4),
        \[
            \S \wr G = \S(\Aut(\S \wr G)) = \S(\Aut(\S) \wr \bar G) = \S(\Aut(\S)) \wr \bar G.
        \]
        By considering the cells of the form $\alpha \oplus \cdots \oplus \alpha$, it follows that $\S = \S(\Aut(\S))$.
        
        (6) Another definition-chase.
    \end{proof}
    
    \section{Set association schemes of degree at most 8}
    
    The coherence axiom (part (2) of \Cref{defn:SAS}) is so stringent that one may not expect nonschurian set association schemes exist.
    The situation appears comparable with that of \emph{superscheme} theory, introduced by Johnson and Smith~\cite{johnson-smith}, which axiomatizes the combinatorial properties of the \emph{ordered} orbital relations of arbitrary arity,
    and the main theorem of superscheme theory is that all superschemes are schurian (see \cite{smith}*{Theorem~8.5}).
    It may come as a surprise therefore that nonschurian set association schemes do exist.
    Indeed there are several of degree 8.
    
    \begin{proposition}
        \label{prop:nonschurian-enumeration}
        Let $\Omega = [d]$.
        \begin{enumerate}
        \item If $d \leq 7$, all set association schemes on $\Omega$ are schurian.
        \item If $d = 8$, there are eight nonschurian set association schemes up to isomorphism, summarized in \Cref{table:nonschurian8}.
        \end{enumerate}
    \end{proposition}
    
    \begin{table}[ht]
        \def\yes{{\color{OliveGreen} Yes}}
        \def\no{{\color{Maroon} No}}
        \begin{center}
            \rowcolors{2}{gray!10}{white}
            \begin{tabular}{ccccccc}
                & $\rank(\S)$ & $\Aut(\S)$ & homogeneous? & vertex-transitive? & fully coherent?\\
                 \hline
                $\S_1$ & 25 & $Q_8 \circ C_4$ & \yes & \yes & \yes\\
                $\S_2$& 30 & $Q_8 \circ C_4$ & \yes & \yes & \yes\\
                $\S_3$ & 28 & $Q_8$ & \yes & \yes & \yes\\
                $\S_4$ & 36 & $Q_8$ & \yes & \yes & \yes\\
                $\S_5$ & 28 & $C_4 \times C_2$ & \yes & \no & \yes\\
                $\S_6$ & 51 & $C_4 \times C_2$ & \no & \no & \yes\\
                $\S_7$ & 43 & $C_4 \times C_2$ & \no & \no & \yes\\
                $\S_8$ & 49 & $C_4 \times C_2$ & \no & \no & \yes\\
            \end{tabular}
        \end{center}
        \caption{Nonschurian set association schemes of degree 8 up to isomorphism}
        \label{table:nonschurian8}
    \end{table}
    
    Notably, the set association scheme $\S = \S_5$, of degree 8 and rank 28, is homogeneous but $\Aut(\S)$ is \emph{intransitive}.
    By \Cref{lem:prop-correspondence}, the corresponding Hamming sandwich $\X = [m]^\S$ has the following properties, for any $m \geq 3$:
    \begin{enumerate}
        \item the degree of $\X$ is $n = m^8$,
        \item $\rank(\X) = 28$,
        \item $\X$ is primitive,
        \item $\Aut(\X)$ is imprimitive,
        \item $|\Aut(\X)| = 8m!^8 > \exp(m) = \exp(n^{1/8})$.
    \end{enumerate}
    Also, since $\Aut(\X) = \Sym(m) \wr \Aut(\S)$, $\mu(\Aut(\X)) = 2m^7$ and $\theta(\Aut(\X)) = m$.
    This proves \Cref{thm:main}, subject to the construction of $\S_5$.
    
    Additionally we may form the wreath product $\S = \S_5 \wr G$ for any transitive permutation group $G \leq \Sym(k)$ for any $k \geq 1$.
    By \Cref{lem:wreath-product}, $\S$ is again homogeneous and vertex-intransitive,
    but of larger degree and rank.
    Thus we have the following corollary.
    
    \begin{corollary}
        There are vertex-intransitive homogeneous set association schemes of arbitrarily large rank.
    \end{corollary}
    
    \def\sage{\texttt{sage}}
    \def\gap{\texttt{GAP}}
    
    In the rest of this section we describe the construction of $\S_1, \dots, \S_8$
    and sketch the verification of \Cref{prop:nonschurian-enumeration}.
    We do not prove the coherence of $\S_1, \dots, \S_8$ by hand but rely on a computer-check.
    All calculations were done in \sage/\gap.
    Source code may be found at \href{https://github.com/seaneberhard/sassy}{https://github.com/seaneberhard/sassy}.
    
    \subsection{Constructions}
    
    Each of the constructions we give has the following general form.
    Let $H \leq G \leq \Sym(d)$ be permutation groups such that $[G:H] = 2$.
    Then $\S(H) \leq \S(G)$ and there are some $k = \rank(\S(H)) - \rank(\S(G))$ cells $\alpha_1, \dots, \alpha_k \in \S(G)$ that split in two in $\S(H)$.
    The partition $\S$ is obtained from $\S(G)$ by splitting the cells in some proper nonempty subset of $\{\alpha_1, \dots, \alpha_k\}$.
    If $\S$ is coherent (there is no reason in general it should be)
    then so is the partition $\S'$ obtained by splitting the cells in the complementary subset instead,
    so in this circumstance we have two nonschurian set association schemes $\S$ and $\S'$ with $\Aut(\S) = \Aut(\S') = H$.
    
    \begin{figure}[ht]
    \[
        \begin{tikzcd}
            G \ar[dd, no head, "{[G:H] = 2}"] &&& \S(G) \ar[dl, no head] \ar[dr, no head] \\
            &&\S \ar[dr, no head] && \S' \ar[dl, no head] \\
            H &&& \S(H)
        \end{tikzcd}
    \]
    \caption{Two nonschurian set association schemes}
    \label{fig:S1S2}
    \end{figure}
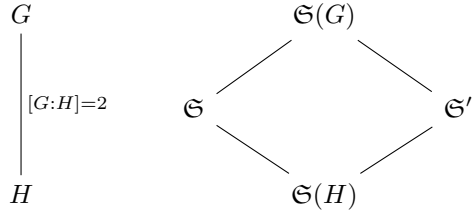
    
    \subsubsection{$\S_1$ and $\S_2$}
    
    Let
    \[
        G_1 = \{ax + b : a \in (\Z/8\Z)^\times, b \in \Z/8\Z\} \cong C_8 \rtimes (C_2 \times C_2)
    \]
    (\texttt{TransitiveGroup(8, 15)} in \gap)
    be the affine group modulo 8 acting naturally on $\Omega = \Z/8\Z$.
    Then $\S(G_1)$ is a schurian scheme of rank 24.
    Let $\alpha_1 = \{0, 1, 2, 3\}^G$.
    Then $\alpha_1$ consists of all arithmetic progressions of length 4 and odd common difference.
    Let $\alpha \subset \alpha_1$ consist of all arithmetic progressions $\{a, a+d, a+2d, a+3d\}$ with $a$ even and $d \in \{1, 5\}$.
    The scheme $\S_1$, of rank 25, is obtained by splitting $\alpha_1$ into $\alpha$ and $\alpha_1 \setminus \alpha$.
    Then $H_1 = \Aut(\S_2)$ is an index-two subgroup of $G_1$ and $\S_2$ is related as described above.
    The group $H_1 = \langle x+2, 3x+1, 5x\rangle$ is internally the central product of $\langle x+2, 3x+1 \rangle \cong Q_8$ and $\langle 5x + 2 \rangle \cong C_4$; it is sometimes called the \emph{Pauli group}.
    
    \subsubsection{$\S_3$ and $\S_4$}
    
    The schemes $\S_3$ and $\S_4$ are obtained in the same way but starting with the quasidihedral group
    \[
        G_2 = \{ax + b : a \in \{1, 3\}, b \in \Z/8\Z\} \leq G_1
    \]
    (\texttt{TransitiveGroup(8, 8)} in \gap).
    We split the same cell $\alpha_1$ as before into the same two cells $\alpha$ and $\alpha_1 \setminus \alpha$ to obtain $\S_3$, and $\S_4$ is related as before.
    Here $H_2 = \Aut(\S_3)$ is the regular subgroup $\langle x+2, 3x+1 \rangle \cong Q_8$.
    
    This construction can be described alternatively as follows.
    Identify $\Omega$ with $Q_8 = \{1, i, j, k, -1, -i, -j, -k\}$ in such a way that $H_2 = Q_8$ in the right regular action.
    Then $H_2$ has two orbits of 4-sets of the form $\{\pm1, \pm i, \pm j, \pm k\}$.
    Fusing those two orbits produces $\S_4$.
    
    \subsubsection{$\S_5$ and $\S_6$}
    
    The schemes $\S_5$ and $\S_6$ are obtained in the same way but starting with the modular 16-group 
    \[
        G_3 = \{ax + b : a \in \{1, 5\}, b \in \Z/8\Z\} \leq G_1
    \]
    (\texttt{TransitiveGroup(8, 7)} in \gap).
    Again we split the same cell $\alpha_1$ into the same two cells $\alpha$ and $\alpha_1 \setminus \alpha$ to obtain $\S_5$, and $\S_6$ is related as before.
    In this case $H_3 = \Aut(\S_5)$ is the intransitive subgroup $\langle x+2, 5x\rangle \cong C_4 \times C_2$.
    
    \subsubsection{$\S_7$ and $\S_8$}
    
    In this case we start with the intransitive group
    \[
        G_4 = \langle (1,3,5,7)(2, 4, 6, 8), (2,4)(6,8) \rangle \cong (C_4 \times C_2) \rtimes C_2.
    \]
    The scheme $\S_7$ is derived from $\S(G_4)$ by separating $\alpha$, $\alpha^c$, and $\beta$, where
    \begin{gather*}
        \alpha = \{\{1, 2, 3\},
         \{1, 3, 6\},
         \{1, 4, 7\},
         \{1, 7, 8\},
         \{2, 5, 7\},
         \{3, 4, 5\},
         \{3, 5, 8\},
         \{5, 6, 7\}\},\\
         \beta = \{\{1, 2, 3, 6\}, \{1, 4, 7, 8\}, \{2, 5, 6, 7\}, \{3, 4, 5, 8\}\}.
    \end{gather*}
    Then $H_4 = \Aut(\S_7) = \langle (1,3,5,7)(2,4,6,8), (2,6)(4,8) \rangle \cong C_4 \times C_2$.
    
    \subsubsection{Two of degree 9}
    
    There are at least two nonschurian set association schemes of degree 9.
    Let
    \[
        G = \{ax + b : a \in (\Z/9\Z)^\times, b \in \Z/9\Z\} \cong C_9 \rtimes C_6
    \]
    (\texttt{TransitiveGroup(9, 10)} in \gap)
    be the affine group modulo 9 acting naturally on $\Omega = \Z/9\Z$.
    Let $H \leq G$ be the subgroup defined by restricting $a$ to $\{1, 4, 7\}$ (the modular 27-group).
    A nonschurian set association scheme is obtained from $\S(G)$ by splitting $\{1,2,3,5\}^G$ into $\{1,2,3,5\}^H$ and $\{1,2,3,6\}^H$, and another is related as in \Cref{fig:S1S2}.
    
    \begin{table}[hhhh]
        \def\yes{{\color{OliveGreen} Yes}}
        \def\no{{\color{Maroon} No}}
        \begin{center}
            \rowcolors{2}{gray!10}{white}
            \begin{tabular}{cccccccc}
                & $\rank(\S)$ & $\Aut(\S)$ & homogeneous? & vertex-transitive? & fully coherent?\\
                 \hline
                 $\S$ & 24 & $C_9 \rtimes C_3$ & \yes & \yes &\yes\\
                 $\S'$ & 26 & $C_9 \rtimes C_3$ & \yes & \yes &\yes\\
            \end{tabular}
        \end{center}
        \caption{Two nonschurian set association schemes of degree 9}
        \label{table:nonschurian9}
    \end{table}

    \subsection{Computational aspects}
    
    Checking coherence of a given partition $\S$ of $2^{[d]}$ naively requires inspecting $(2^d)^3$ triples $(a, b, c) \in (2^{[d]})^3$ and sorting them according to the value of $(\chi(a), \chi(b), \chi(c), |a \cap b|, |a \cap c|, |b \cap c|, |a \cap b \cap c|)$, where $\chi : 2^{[d]} \to \{0, \dots, \rank(\S) - 1\}$ is a coloring function representing $\S$. For $d = 8$, $(2^d)^3 \approx 1.7 \times 10^7$, so this is feasible on any modern computer.
    Thus checking coherence of $\S_1, \dots, \S_8$ is straightforward,
    but checking that there are no others requires further comment.
    
    The naive algorithm described above for checking coherence in general actually outputs a refinement $\S' \leq \S$ such that $\S' = \S$ if and only if $\S$ is coherent, and furthermore such that if $\S_0 \leq \S$ is any coherent partition then $\S_0 \leq \S'$.
    Iterating this algorithm therefore produces the coarsest coherent refinement of $\S$.
    This is a natural variant for set association schemes of the \emph{Weisfeiler--Leman} (WL) algorithm for coherent configurations, so we call it that,
    and the coarsest coherent refinement of $\S$ is called its \emph{WL stabilization}.
    
    In practice, it helps to compute $\Aut(\S)$ first. Then if $\S = \S(\Aut(\S))$ then $\S$ is schurian and therefore coherent. In general the WL algorithm does not alter $\Aut(\S)$ and it can be sped up by a factor of $|\Aut(\S)|$ by only enumerating those triples $(a, b, c)$ with $a$ in a given set of $\Aut(\S)$-orbit representatives.
    Enumerating $\Aut(\S)$-orbit representatives in $2^{[d]}$ can be done quickly using a union--find data structure.
    
    Now given an arbitrary coherent partition $\S$ of $2^{[d]}$, we can enumerate some coherent strict refinements $\S_1, \dots, \S_n$ of $\S$ such that any coherent refinement of $\S$ refines at least one of $\S_1, \dots, \S_n$, as follows.
    Write
    \[
        \S = \{\alpha_0, \dots, \alpha_r, \alpha_r^c, \dots, \alpha_0^c\}
    \]
    (there will be duplicates here if any $\alpha_i = \alpha_i^c$), where
    \[
        0 = |a_0| \leq \cdots \leq |a_r| = \floor{d/2} \qquad (a_i \in \alpha_i).
    \]
    Now for $i = 0, \dots, r$ and each subset $\alpha \subset \alpha_i$ such that $0 < |\alpha| \leq |\alpha_i| / 2$,
    let $\S[\alpha]$ be the WL stabilization of the partition formed from $\S$ by splitting $\alpha_i$ into $\alpha$ and $\alpha_i \setminus \alpha$, and, if it is not already in the list, name it $\S_k$ for the next available index $k$.
    
    In fact, we need only consider the subsets $\alpha \subset \alpha_i$ which are ``design-like'' with respect to $\alpha_j$ for $j < i$, in the sense that the bipartite graph
    \[
        \{(a, b) \in \alpha_j \times \alpha : a \subset b\}
    \]
    is biregular, because if $\alpha$ were not design-like with respect to $\alpha_j$ then the WL algorithm would split $\alpha_j$, and hence $\S[\alpha]$ already refines some previous $\S_k$.
    Design-like $\alpha$ can be enumerated using either backtracking or integer programming (this is the most prohibitive part of the process).
    
    Moreover, it suffices to consider only one $\alpha$ per weak-automorphism-orbit, since $\S[\alpha] \cong \S[\alpha']$ for $\alpha, \alpha'$ in the same orbit.
    For example if $\S$ is trivial and $i=2$, arbitrary subsets of $\alpha_2$ correspond to arbitrary undirected graphs on $[d]$, but it suffices to consider only regular graphs up to isomorphism, of which there many fewer (though still many).
    
    By starting with the trivial scheme and iterating, we end up enumerating the entire lattice of set association schemes up to isomorphism, including any nonschurian schemes.
    The whole process for $d \leq 8$ took a few days on the author's laptop.

    \section{Cameron sandwiches}
    
    If $\X$ is a partition of $( \binom{[m]}{k}^d )^2$ such that $\J(m, k)^d \leq \X$ then there is a partition $\S$ of $\{0, \dots, k\}^d$ such that $\X = \{R_\alpha : \alpha \in \S\}$, where
    \[
        (u, v) \in R_\alpha \iff (|u_1 \setminus v_1|, \dots, |u_d \setminus v_d|) \in \alpha.
    \]
    If $\X$ and $\S$ are related in this way we write $\X = \J(m, k)^\S$.
    Clearly $\X$ is a configuration if and only if $\{(0, \dots, 0)\} \in \S$,
    and $\X \leq \Cam(m, k, d)$ if and only if
    \begin{equation}
        \label{eq:sandwich-condition}
        a, a' \in \alpha \in \S \implies a^{\Sym(d)} = (a')^{\Sym(d)}.
    \end{equation}
    If $\J(m, k)^d \leq \X \leq \Cam(m, k, d)$ and $\X$ is coherent then we call $\X$ a \emph{Cameron sandwich}.
    In this section we develop some theory of Cameron sandwiches generalizing that of Hamming sandwiches (even though no nonschurian examples are known for $k > 1$).
    
    The vertices of the Johnson scheme $\J(m, k)$ are $k$-subsets $v \subset [m]$, and the color of a pair $(u, v)$ is by definition $|u \setminus v| \in \{0, \dots, k\}$.
    Note that $|u \setminus v| = |v \setminus u|$, since $|u| = |v| = k$.
    The hypothesis $k \leq m/2$ ensures that all colors appear.
    The structure constants are
    \begin{equation}
        \label{eq:p-formula}
        p^a_{bc}(m) = \sum_i \binom{k-a}{i} \binom{a}{k-b-i} \binom{a}{k-c-i} \binom{m-k-a}{b + c + i - k},
    \end{equation}
    where the summation extends over all integers and binomial coefficients are considered zero when out of bounds.
    Note that $p^a_{bc}$ is a polynomial in $m$ with rational coefficients.
    
    \begin{lemma}
        \label{lem:delta-condition}
        For $a, b, c \in \{0, \dots, k\}$,
        $p^a_{bc}(m) > 0$ if and only if
        \begin{equation}
            \label{eq:delta}
            a \leq b + c, \quad
            b \leq a + c, \quad
            c \leq a + b,
            \tag{$\Delta$}
        \end{equation}
        and $a + b + c \leq m$.
    \end{lemma}
    \begin{proof}
        If $p^a_{bc}(m) > 0$ then there are sets $u, v, w \subset [m]$ of size $k$ such that $|u \setminus v| = a$, $|u \setminus w| = b$, and $|w \setminus v| = c$. Since
        $u \setminus v \subset (u \setminus w) \cup (w \setminus v)$,
        we must have $a \leq b + c$, and likewise for other permutations, so \eqref{eq:delta} holds.
        Since $u \setminus v$, $v \setminus w$, and $w \setminus u$ are pairwise disjoint, we also must have $a + b + c \leq m$.
        
        Conversely, suppose \eqref{eq:delta} holds and $m \geq a + b + c$.
        Without loss of generality we may assume $a \geq b \geq c$.
        Taking the $i = \max(0, k - b - c)$ in \eqref{eq:p-formula} shows that $p^a_{bc}(m) > 0$.
    \end{proof}
    
    We now doubly overload the notation $p^a_{bc}$ as follows:
    \begin{enumerate}
        \item For $a, b, c \in \{0, \dots, k\}^d$, $d \geq 1$, define
        \[
            p^a_{bc} = p^{a_1}_{b_1c_1} \cdots p^{a_d}_{b_dc_d},
        \]
        and observe that $p^a_{bc}$ is again a polynomial in $m$.
        These are the structure constants of $\J(m, k)^d$.
        \item For $\beta, \gamma \subset \{0, \dots, k\}^d$, define
        \[
            p^a_{\beta\gamma} = \sum_{(b, c) \in \beta \times \gamma} p^a_{bc}.
        \]
    \end{enumerate}
    We say $a, b, c \in \{0, \dots, k\}^d$ form a \emph{triangle} if $p^a_{bc}$ is not the zero polynomial.
    By \Cref{lem:delta-condition}, this is the case if and only if $a_i, b_i, c_i$ satisfy \eqref{eq:delta} for each $i$.
    If $x \in \{0, \dots, k\}$ we write $(x)^d$ as shorthand for $(x, \dots, x)$.
    
    The following lemma is self-evident at this point.
    
    \begin{lemma}
        Assume $m \geq 2k > 0$.
        Let $\X = \J(m, k)^\S$ for some partition $\S$ of $\{0, \dots, k\}^d$.
        Then $\X$ is a Cameron sandwich if and only if \eqref{eq:sandwich-condition} holds and $p^a_{\beta\gamma}(m) = p^{a'}_{\beta\gamma}(m)$ for all $\alpha,\beta,\gamma \in \S$ and $a, a' \in \alpha$.
    \end{lemma}

    Properties of $\X = \J(m, k)^\S$ are related to corresponding properties of $\S$, analogously to \Cref{lem:prop-correspondence}.
    Let $\rank(\S)$ denote the number of cells of $\S$.
    We call $\S$ \emph{homogeneous} if $(0, \dots, 0, 1)^{\Sym(d)} \in \S$.
    The \emph{(strong) automorphism group} of $\S$ is the group $\Aut(\S) \leq \Sym(d)$ of permutations of $[d]$ preserving each cell of $\S$.
    If $G \leq \Sym(d)$ we write $\S_k(G)$ for the partition $\{0, \dots, k\}^d / G$ of $\{0, \dots, k\}^d$ into $G$-orbits.
    We call $\S$ \emph{schurian} if $\S = \S_k(\Aut(\S))$.
    
    \begin{lemma}
        \label{lem:prop-correspondence-cameron-sandwich}
        Let $\X = \J(m, k)^\S$ be a Cameron sandwich, where $m \geq 2k > 0$.
        \begin{enumerate}
            \item $\rank(\X) = \rank(\S)$.
            \item $\X$ is primitive $\iff$ $\S$ is homogeneous and $m \geq 2k+1$.
            \item $\Aut(\X) = \Sym(m)^{(k)} \wr \Aut(\S)$.
            \item $\X$ is schurian $\iff$ $\S$ is schurian.
            \item $\Aut(\X)$ is primitive $\iff$ $\Aut(\S)$ is transitive and $m \geq 2k+1$.
        \end{enumerate}
    \end{lemma}
    \begin{proof}
        (1) Immediate from the definition of $\J(m, k)^\S$.
        
        (2) Suppose $\X = \J(m, k)^\S$ is primitive. Let $a \in \alpha \in \S$ where $a = (0, \dots, 0, 1)$.
        Then $\alpha \subset (0, \dots, 0, 1)^{\Sym(d)}$.
        Since the color graph of $\J(m, k)^\S$ defined by $\alpha$ is connected, we must have $\alpha = (0, \dots, 0, 1)^{\Sym(d)}$.
        Hence $\S$ is homogeneous.
        If $m = 2k$ then the color graph of $\J(m, k)^\S$ defined by $\omega = \{(k)^d\}$ is disconnected (all components have just two vertices)
        so we must have $m \geq 2k + 1$.
        
        Conversely, suppose $\S$ is homogeneous and $m \geq 2k+1$.
        Let $\alpha = (0, \dots, 0, 1)^{\Sym(d)}$
        and let $\beta \in \S$, $\beta \neq \{(0)^d\}$.
        Then for any $b \in \beta$ there is some $a \in \alpha$ such that $(a, b, b)$ is a triangle.
        Since $m \geq 2k+1$, \Cref{lem:delta-condition} implies that $p^\alpha_{\beta\beta}(m) > 0$.
        Now the color graph of $\J(m, k)^\S$ defined by $\alpha$ is the Cameron graph, which is connected,
        so the color graph of $\J(m, k)^\S$ defined by $\beta$ is also connected.
        
        (3--5) Exactly as in the proof of \Cref{lem:prop-correspondence}.
    \end{proof}
    
    We are most interested when $\J(m, k)^\S$ is a Cameron sandwich for all $m$. Thus we make the following definition.
    
    \begin{definition}
        \label{defn:VAS}
        Let $\S$ be a partition of $\{0, \dots, k\}^d$.
        We call $\S$ a \emph{vector association scheme} if
        $\J(m, k)^\S$ is a Cameron sandwich for all $m\geq 2k$.
        Equivalently,
        \begin{enumerate}
            \item $a^{\Sym(d)} = (a')^{\Sym(d)}$ for all $a, a' \in \alpha \in \S$,
            \item $p^a_{\beta\gamma} = p^{a'}_{\beta\gamma}$ for all $\alpha, \beta, \gamma \in \S$ and $a, a' \in \alpha$.
        \end{enumerate}
    \end{definition}
        
    \begin{remark}
        \label{rem:vas-generalizes-sas}
        Vector association schemes with $k=1$ are equivalent to set association schemes.
        This follows from \Cref{lem:hamming-sandwiches}(2).
    \end{remark}
    
    We can and will denote the common value of $p^a_{\beta\gamma}$ for $a \in \alpha$ by $p^\alpha_{\beta\gamma}$. The polynomials $p^\alpha_{\beta\gamma}$ are the \emph{structure polynomials} of $\S$.
    
    As for coherent configurations and set association schemes, the prototypical vector assocation schemes are the orbital schemes $\S_k(G)$ for $G \leq \Sym(d)$.
    Among these we have the \emph{discrete scheme} $\S_k(1)$ and the \emph{trivial scheme} $\S_k(\Sym(d))$.
    
    \begin{lemma}
        \label{lem:large-cameron-sandwich-implies-VAS}
        For all $k, d \geq 1$ there is some $m_0 = m_0(k, d)$ such that if $\X = \J(m, k)^\S$ is a Cameron sandwich for some $m \geq m_0$ then $\S$ is a vector association scheme.
    \end{lemma}
    \begin{proof}
        If $\S$ is not a vector assocation scheme then there are some $\alpha, \beta, \gamma \in \S$ and $a, a' \in \alpha$ such that $p^a_{\beta\gamma} - p^{a'}_{\beta\gamma}$ is not the zero polynomial.
        Therefore there is some $m_1 = m_1(\S)$ such that $p^a_{\beta\gamma}(m) - p^{a'}_{\beta\gamma}(m) \neq 0$ for all $m \geq m_1$.
        Let $m_0$ be the maximum of $m_1(\S)$ over all choices of $\S$.
    \end{proof}
    
    \begin{lemma}
        \label{lem:basic-props-VAS}
        Let $\S$ be a vector association scheme on $\{0, \dots, k\}^d$.
        \begin{enumerate}
            \item $\alpha \in \S \implies\alpha^c \in \S$, where $\alpha^c = \{(k)^d - a : a \in \alpha\}$,
            \item For any $\alpha, \beta \in \S$, the domination graph $\{(a, b) \in \alpha \times \beta : a \leq b\}$ is biregular, where $a \leq b$ means $a_i \leq b_i$ for all $i$.
            \item If $\S$ is homogeneous and $\beta \in \S$ then $\sum \beta = (x)^d$ for some $x \geq 0$.
        \end{enumerate}
    \end{lemma}
    \begin{proof}
        Let $\omega = \{(k)^d\}$. By condition (1) of \Cref{defn:VAS}, $\omega \in \S$.
    
        (1) Let $a \in \alpha \in \S$ and let $\beta$ be the cell of $\S$ containing $b = (k)^d - a$.
        Since $a, b, (k)^d$ form a triangle, $p^a_{\beta\omega} \neq 0$.
        Hence for any $a' \in \alpha$, $p^{a'}_{\beta\omega} \neq 0$, so there is some $b' \in \beta$ such that $a', b', (k)^d$ forms a triangle.
        This is only possible if $b' = (a')^c$. Hence $\alpha^c \subset \beta$, and likewsie $\beta^c \subset \alpha$, so $\beta = \alpha^c$.
        
        (2) Let $(a_0, b_0) \in \alpha \times \beta$ be such that $a_0 \leq b_0$.
        Let $c_0 = b_0 - a_0$ and let $\gamma$ be the cell of $\S$ containing $c_0$.
        Since $a_0 + c_0 = b_0$, the only $(a, b, c) \in a_0^{\Sym(d)} \times b_0^{\Sym(d)} \times c_0^{\Sym(d)}$ such that $(a, b, c)$ forms a triangle are those
        for which $a + c = b$.
        In this case, from \eqref{eq:p-formula}, the leading term of the polynomial $p^a_{bc}$ is
        \[
            \lambda(p^a_{bc}) = \prod_{i=1}^d \frac{(k-a_i)!}{(k-b_i)! c_i!^2} m^{c_i}.
        \]
        By \Cref{defn:VAS}(1), this monomial depends only on $\alpha, \beta, \gamma$.
        Hence if $(a, b, c) \in \alpha \times \beta \times \gamma$ is a triangle then $\lambda(p^a_{bc}) = \lambda(p^{a_0}_{b_0c_0})$.
        Therefore, for $a \in \alpha$,
        \[
            \lambda(p^a_{\beta\gamma}) = N^a_{\beta\gamma} \lambda(p^{a_0}_{b_0c_0}),
        \]
        where $N^a_{\beta\gamma}$ is the number of $(b, c) \in \beta \times \gamma$ such that $a + c = b$.
        Hence $N^a_{\beta\gamma} = N^{a'}_{\beta\gamma}$ for all $a, a' \in \alpha$.
        Since $N^a_{\beta\gamma} = \#\{b \in \beta : b - a \in \gamma\}$,
        summing over all possibilities for $\gamma \in \S$ (containing some $b_0 - a_0$ for $(a_0, b_0) \in \alpha \times \beta$) gives $\#\{b \in \beta : a \leq b\}$.
        Hence this number is independent of $a \in \alpha$,
        and by applying (1) we also obtain that $\#\{a \in \alpha : a \leq b\}$ is independent of $b \in \beta$,
        so we have biregularity.
        
        (3)
        Let $\alpha_t = (0, \dots, 0, t)^{\Sym(d)}$ for $1 \leq t \leq k$.
        Since $\S$ is homogeneous, $\alpha_1 \in \S$,
        so $\alpha_t \in \S$ for each $t$ by (2).
        Now
        \[
            \br{\sum \beta}_i = \sum_{b \in \beta} b_i = \sum_{t=1}^k \#\{b \in \beta : t \leq b_i\}.
        \]
        Hence (3) follows from
        another application of (2).
    \end{proof}
    
    \begin{lemma}
        \label{lem:d=2}
        Let $\S$ be a homogeneous vector association scheme on $\{0, \dots, k\}^d$ with $d \leq 2$.
        Then $\S$ is trivial.
    \end{lemma}
    \begin{proof}
        If $d = 1$ the lemma is trivial by condition (1) of \Cref{defn:VAS},
        so assume $d = 2$.
        Let $(x, y) \in \alpha \in \S$.
        Then $\alpha \subset \{(x, y), (y, x)\}$.
        If $x \neq y$, \Cref{lem:basic-props-VAS}(3) implies $\alpha = \{(x, y), (y, x)\} = (x, y)^{\Sym(2)}$.
        Hence $\S$ is trivial.
    \end{proof}
    
    \begin{proposition}
        \label{prop:kd<8}
        Let $\S$ be a homogeneous vector association scheme on $\{0, \dots, k\}^d$ with $kd < 8$.
        Then $\S$ is schurian.
    \end{proposition}
    \begin{proof}
        If $k = 1$ then this follows from \Cref{rem:vas-generalizes-sas}
        and \Cref{prop:nonschurian-enumeration}.
        If $d \leq 2$ then this follows from the previous lemma.
        This leaves the case $(k, d) = (2, 3)$, which is shown in \Cref{fig:kd23}.
                
        \begin{figure}[hhh]
            \begin{tikzpicture}[scale=2]
                \foreach \x in {0, 1, 2}{
                    \node at (\x, -0.2, 0) {\x};
                    \node at (-0.2, \x, 0) {\x};
                    \node at (-0.2, 0, -\x) {\x};
                    \foreach \y in {0, 1, 2}{
                        \draw (\x, \y, 0) -- (\x, \y, -2);
                        \draw (\x, 0, -\y) -- (\x, 2, -\y);
                        \draw (0, \x, -\y) -- (2, \x, -\y);
                    }
                }
                \foreach \color/\points in {
                    white/{(0,0,0)},
                    gray/{(1,1,-1)},
                    black/{(2,2,-2)},
                    red/{(1,0,0), (0,1,0), (0,0,-1)},
                    green/{(2,0,0), (0,2,0), (0,0,-2)},
                    cyan/{(1,2,-2), (2,1,-2), (2,2,-1)},
                    magenta/{(0,2,-2), (2,0,-2), (2,2,0)},
                    blue/{(1,1,0), (1,0,-1), (0,1,-1)},
                    yellow/{(1,1,-2), (1,2,-1), (2,1,-1)},
                }
                    \foreach \point in \points
                        \draw[fill=\color] \point circle (2pt);
            \end{tikzpicture}
            \caption{Part of a vector association scheme on $\{0, 1, 2\}^3$}
            \label{fig:kd23}
        \end{figure}
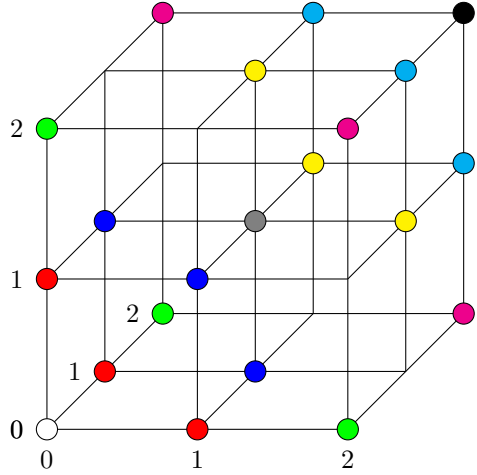
        
        Let $S = \Sym(3)$.
        For each orbit of $S$ on $\{0, 1, 2\}^3$, check whether there is a proper nonempty subset $\beta$ such that $\sum \beta = (x,x,x)$ for some $x$.
        If not, the full orbit must be a cell of $\S$ by \Cref{lem:basic-props-VAS}(3).
        That leaves only the orbit $(0, 1, 2)^S$ to consider.
        Let $\beta \subset (0,1,2)^S$ be a cell of $\S$,
        and without loss of generality suppose $(0,1,2) \in \beta$.
        If $|\beta| = 6$ then $\S$ is trivial.
        If $|\beta| = 2$ then $\beta = \{(0,1,2),(2,1,0)\}$,
        but then \Cref{lem:basic-props-VAS}(2) with $\alpha = (0,0,2)^S$ gives a contradiction.
        The last possibility is that the restriction $\S | (0,1,2)^S$ is the partition
        \[
            \{
                \{(0,1,2), (2,0,1), (1,2,0)\},
                \{(1,0,2), (2,1,0), (0,2,1)\}
            \},
        \]
        and in this case $\S=\S_2(\Alt(3))$.
    \end{proof}
    
    Now we can prove \Cref{prop:nonschurian1/8}, restated here.
    
    \begin{proposition}
        Assume \Cref{conj:sand}.
        Let $\X$ be a nonschurian PCC.
        Then
        \[
            |\Aut(\X)| \leq \exp O(n^{1/8} \log n).
        \]
    \end{proposition}
    \begin{proof}
        If $|\Aut(\X)| \geq \exp(n^{1/8})$ and $n$ is sufficiently large (which we may assume) then \Cref{conj:sand} implies $\J(m, k)^d \leq \X \leq \Cam(m, k, d)$ for some $m, k, d$ with $m \geq 2k > 0$.
        Then
        \begin{equation}
            \label{eq:aut-bound}
            |\Aut(\X)| \leq |\Aut(\Cam(m, k, d))| = m!^d d! \leq \exp O(dm \log m + d \log d).
        \end{equation}
        Since $n = \binom{m}{k}^d \geq (m/k)^{kd}$ and $k, d \leq \log_2 n$, we deduce
        \[
            n^{1/8} \leq \log |\Aut(\X)| \leq O(n^{1/kd} (\log n)^3).
        \]
        Since we may assume that $n$ is large, it follows that $kd \le 8$.
        Hence $m$ is large and \Cref{lem:prop-correspondence-cameron-sandwich,lem:large-cameron-sandwich-implies-VAS} imply $\X = \J(m, k)^\S$ for some nonschurian homogeneous vector association scheme $\S$.
        By the previous proposition we must have $kd = 8$ so we get the claimed bound on $|\Aut(\X)|$ from \eqref{eq:aut-bound}.
    \end{proof}
    
    To end, we prove \Cref{prop:ranks}, restated here.
        
    \begin{proposition}
        Assume \Cref{conj:sand}.
        Let $\X$ be a nonschurian PCC of rank less than $12$.
        Then
        \[
            |\Aut(\X)| \leq \exp O((\log n)^{O(1)}).
        \]
    \end{proposition}
    \begin{proof}
        By \Cref{conj:sand} we may assume $\J(m,k)^d \le \X \le \Cam(m,k,d)$ for some $m,k,d$ with $m \ge 2k > 0$.
        In particular $\rank(\X) \ge \rank(\Cam(m,k,d)) = \binom{k+d}{k}$.
        Since $\rank(\X)$ is bounded it follows that $k$ and $d$ are bounded, so $m$ is large (since we may assume $n$ is large), and hence \Cref{lem:prop-correspondence-cameron-sandwich,lem:large-cameron-sandwich-implies-VAS} imply $\X= \J(m,k)^\S$ for some nonschurian homogeneous vector association scheme $\S$ on $\{0,\dots,k\}^d$.
        By \Cref{lem:d=2}, $d \ge 3$.
        Since $\binom{k+d}{k} < 12$, either $(k, d) = (2, 3)$ or $k = 1$.
        The case $(k,d) = (2,3)$ is impossible by \Cref{prop:kd<8}, so $k = 1$ and $\S$ can be identified with a set association scheme of degree $d$.
        For $d \le 8$ we have \Cref{prop:nonschurian-enumeration}, while any nontrivial set association scheme of degree at least $9$ has rank at least $12$.
    \end{proof}

%
    
    \bibliography{refs}

\begin{bibdiv}
\begin{biblist}

\bib{babai-uniprimitive}{article}{
      author={Babai, L\'{a}szl\'{o}},
       title={On the order of uniprimitive permutation groups},
        date={1981},
        ISSN={0003-486X},
     journal={Ann. of Math. (2)},
      volume={113},
      number={3},
       pages={553\ndash 568},
         url={https://doi.org/10.2307/2006997},
      review={\MR{621016}},
}

\bib{babai-2-trans}{article}{
      author={Babai, L\'{a}szl\'{o}},
       title={On the order of doubly transitive permutation groups},
        date={1981/82},
        ISSN={0020-9910},
     journal={Invent. Math.},
      volume={65},
      number={3},
       pages={473\ndash 484},
         url={https://doi.org/10.1007/BF01396631},
      review={\MR{643565}},
}

\bib{babai-srg-1}{inproceedings}{
      author={Babai, L\'{a}szl\'{o}},
       title={On the automorphism groups of strongly regular graphs {I}},
        date={2014},
   booktitle={I{TCS}'14---{P}roceedings of the 2014 {C}onference on
  {I}nnovations in {T}heoretical {C}omputer {S}cience},
   publisher={ACM, New York},
       pages={359\ndash 368},
         url={https://doi.org/10.1145/2554797.2554830},
      review={\MR{3359489}},
}

\bib{babai-GI}{article}{
      author={Babai, L\'{a}szl\'{o}},
       title={{Graph Isomorphism in Quasipolynomial Time}},
        date={2015},
       pages={\href{https://arxiv.org/abs/1512.03547v2}{arXiv:1512.03547v2}},
}

\bib{babai-srg-2}{article}{
      author={Babai, L\'{a}szl\'{o}},
       title={On the automorphism groups of strongly regular graphs {II}},
        date={2015},
        ISSN={0021-8693},
     journal={J. Algebra},
      volume={421},
       pages={560\ndash 578},
         url={https://doi.org/10.1016/j.jalgebra.2014.09.007},
      review={\MR{3272397}},
}

\bib{babai-ICM2018}{inproceedings}{
      author={Babai, L\'{a}szl\'{o}},
       title={Group, graphs, algorithms: the graph isomorphism problem},
        date={2018},
   booktitle={Proceedings of the {I}nternational {C}ongress of
  {M}athematicians---{R}io de {J}aneiro 2018. {V}ol. {IV}. {I}nvited lectures},
   publisher={World Sci. Publ., Hackensack, NJ},
       pages={3319\ndash 3336},
      review={\MR{3966534}},
}

\bib{bailey}{book}{
      author={Bailey, R.~A.},
       title={Association schemes},
      series={Cambridge Studies in Advanced Mathematics},
   publisher={Cambridge University Press, Cambridge},
        date={2004},
      volume={84},
        ISBN={0-521-82446-X},
         url={https://doi.org/10.1017/CBO9780511610882},
        note={Designed experiments, algebra and combinatorics},
      review={\MR{2047311}},
}

\bib{BannaiBannaiItoTanaka+2021}{book}{
      author={Bannai, Eiichi},
      author={Bannai, Etsuko},
      author={Ito, Tatsuro},
      author={Tanaka, Rie},
       title={Algebraic combinatorics},
   publisher={De Gruyter},
        date={2021},
        ISBN={9783110630251},
         url={https://doi.org/10.1515/9783110630251},
}

\bib{bannai-ito}{book}{
      author={Bannai, Eiichi},
      author={Ito, Tatsuro},
       title={Algebraic combinatorics. {I}},
   publisher={The Benjamin/Cummings Publishing Co., Inc., Menlo Park, CA},
        date={1984},
        ISBN={0-8053-0490-8},
        note={Association schemes},
      review={\MR{882540}},
}

\bib{bose--shimamoto}{article}{
      author={Bose, R.~C.},
      author={Shimamoto, T.},
       title={Classification and analysis of partially balanced incomplete
  block designs with two associate classes},
        date={1952},
        ISSN={0162-1459},
     journal={J. Amer. Statist. Assoc.},
      volume={47},
       pages={151\ndash 184},
  url={http://links.jstor.org/sici?sici=0162-1459(195206)47:258<151:CAAOPB>2.0.CO;2-G&origin=MSN},
      review={\MR{48772}},
}

\bib{cameron}{article}{
      author={Cameron, Peter~J.},
       title={Finite permutation groups and finite simple groups},
        date={1981},
        ISSN={0024-6093},
     journal={Bull. London Math. Soc.},
      volume={13},
      number={1},
       pages={1\ndash 22},
         url={https://doi.org/10.1112/blms/13.1.1},
      review={\MR{599634}},
}

\bib{CP}{article}{
      author={Chen, Gang},
      author={Ponomarenko, Ilia},
       title={Lectures on coherent configurations},
        date={2019},
      eprint={http://www.pdmi.ras.ru/~inp/ccNOTES.pdf},
        note={Online lecture notes},
}

\bib{dixon--mortimer}{book}{
      author={Dixon, John~D.},
      author={Mortimer, Brian},
       title={Permutation groups},
      series={Graduate Texts in Mathematics},
   publisher={Springer-Verlag, New York},
        date={1996},
      volume={163},
        ISBN={0-387-94599-7},
         url={https://doi.org/10.1007/978-1-4612-0731-3},
      review={\MR{1409812}},
}

\bib{followup}{misc}{
      author={Eberhard, Sean},
      author={Muzychuk, Misha},
       title={Fusions of tensor powers of {J}ohnson schemes},
        note={In preparation},
}

\bib{godsil}{article}{
      author={{Godsil}, Chris},
       title={{Generalized Hamming schemes}},
        date={2010},
       pages={\href{https://arxiv.org/abs/1011.1044}{arXiv:1011.1044}},
}

\bib{higman}{article}{
      author={Higman, D.~G.},
       title={Coherent configurations. {I}},
        date={1970},
        ISSN={0041-8994},
     journal={Rend. Sem. Mat. Univ. Padova},
      volume={44},
       pages={1\ndash 25},
         url={http://www.numdam.org/item?id=RSMUP_1970__44__1_0},
      review={\MR{325420}},
}

\bib{johnson-smith}{article}{
      author={Johnson, K.~W.},
      author={Smith, J. D.~H.},
       title={Characters of finite quasigroups. {IV}. {P}roducts and
  superschemes},
        date={1989},
        ISSN={0195-6698},
     journal={European J. Combin.},
      volume={10},
      number={3},
       pages={257\ndash 263},
         url={https://doi.org/10.1016/S0195-6698(89)80060-5},
      review={\MR{1029172}},
}

\bib{kivva1}{article}{
      author={Kivva, Bohdan},
       title={A characterization of {J}ohnson and {H}amming graphs and proof of
  {B}abai's conjecture},
        date={2021},
        ISSN={0095-8956},
     journal={J. Combin. Theory Ser. B},
      volume={151},
       pages={339\ndash 374},
         url={https://doi.org/10.1016/j.jctb.2021.07.003},
      review={\MR{4294227}},
}

\bib{kivva-arxiv}{article}{
      author={{Kivva}, Bohdan},
       title={{On the automorphism groups of rank-4 primitive coherent
  configurations}},
        date={2021},
       pages={\href{https://arxiv.org/abs/2110.13861}{arXiv:2110.13861}},
}

\bib{liebeck}{article}{
      author={Liebeck, Martin~W.},
       title={On minimal degrees and base sizes of primitive permutation
  groups},
        date={1984},
        ISSN={0003-889X},
     journal={Arch. Math. (Basel)},
      volume={43},
      number={1},
       pages={11\ndash 15},
         url={https://doi.org/10.1007/BF01193603},
      review={\MR{758332}},
}

\bib{liebeck--saxl}{article}{
      author={Liebeck, Martin~W.},
      author={Saxl, Jan},
       title={Minimal degrees of primitive permutation groups, with an
  application to monodromy groups of covers of {R}iemann surfaces},
        date={1991},
        ISSN={0024-6115},
     journal={Proc. London Math. Soc. (3)},
      volume={63},
      number={2},
       pages={266\ndash 314},
         url={https://doi.org/10.1112/plms/s3-63.2.266},
      review={\MR{1114511}},
}

\bib{maroti}{article}{
      author={Mar\'{o}ti, Attila},
       title={On the orders of primitive groups},
        date={2002},
        ISSN={0021-8693},
     journal={J. Algebra},
      volume={258},
      number={2},
       pages={631\ndash 640},
         url={https://doi.org/10.1016/S0021-8693(02)00646-4},
      review={\MR{1943938}},
}

\bib{pyber-2-trans}{article}{
      author={Pyber, L.},
       title={On the orders of doubly transitive permutation groups, elementary
  estimates},
        date={1993},
        ISSN={0097-3165},
     journal={J. Combin. Theory Ser. A},
      volume={62},
      number={2},
       pages={361\ndash 366},
         url={https://doi.org/10.1016/0097-3165(93)90053-B},
      review={\MR{1207742}},
}

\bib{schur}{article}{
      author={{Schur}, {I}ssai},
       title={{Z}ur theorie der einfach transitiven {P}ermutationsgruppen},
        date={1933},
     journal={Sitzungsberichte der Berliner Akademie},
       pages={598—623},
        note={Republished in \emph{{G}esammelte {A}bhandlungen {III}},
  Springer, 1973},
}

\bib{smith}{book}{
      author={Smith, Jonathan D.~H.},
       title={An introduction to quasigroups and their representations},
      series={Studies in Advanced Mathematics},
   publisher={Chapman \& Hall/CRC, Boca Raton, FL},
        date={2007},
        ISBN={978-1-58488-537-5; 1-58488-537-8},
      review={\MR{2268350}},
}

\bib{sun-thesis}{article}{
      author={Sun, Xiaorui},
       title={On the isomorphism testing of graphs},
        date={2016},
         url={https://academiccommons.columbia.edu/doi/10.7916/D8416X8N},
        note={Ph.D. thesis,
  \href{https://doi.org/10.7916/D8416X8N}{10.7916/D8416X8N}},
}

\bib{sun-wilmes-abs}{inproceedings}{
      author={Sun, Xiaorui},
      author={Wilmes, John},
       title={Faster canonical forms for primitive coherent configurations
  (extended abstract)},
        date={2015},
   booktitle={S{TOC}'15---{P}roceedings of the 2015 {ACM} {S}ymposium on
  {T}heory of {C}omputing},
   publisher={ACM, New York},
       pages={693\ndash 702},
      review={\MR{3388249}},
}

\bib{sun-wilmes}{article}{
      author={{Sun}, Xiaorui},
      author={{Wilmes}, John},
       title={{Structure and automorphisms of primitive coherent
  configurations}},
        date={2015},
       pages={\href{https://arxiv.org/abs/1510.02195}{arXiv:1510.02195}},
}

\bib{WL}{article}{
      author={Weisfeiler, Boris},
      author={Leman, Andrei},
       title={The reduction of a graph to canonical form and the algebra which
  appears therein},
        date={1968},
     journal={NTI, Series},
      volume={2},
      number={9},
       pages={12\ndash 16},
}

\bib{wielandt}{book}{
      author={Wielandt, Helmut},
       title={Finite permutation groups},
   publisher={Academic Press, New York-London},
        date={1964},
        note={Translated from the German by R. Bercov},
      review={\MR{0183775}},
}

\bib{wilmes-thesis}{book}{
      author={Wilmes, John},
       title={Structure, automorphisms, and isomorphisms of regular
  combinatorial objects},
   publisher={ProQuest LLC, Ann Arbor, MI},
        date={2016},
        ISBN={978-1369-12974-8},
        note={Ph.D. thesis -- The University of Chicago,
  \url{https://www.proquest.com/docview/1837431537}},
      review={\MR{3593268}},
}

\end{biblist}
\end{bibdiv}
\end{document}